\documentclass{amsart}
\usepackage{amssymb}
\usepackage{amsmath}

\newtheorem{theo}{Theorem}[section]
\newtheorem{lemm}[theo]{Lemma}
\newtheorem{defi}[theo]{Definition}

\newtheorem{prop}[theo]{Proposition}
\newtheorem{rema}[theo]{Remark}

\def\la{\langle}
\def\ra{\rangle}

\newcommand\NN{{\mathbb N}}
\newcommand\RR{{{\mathbb R}}}
\newcommand\ZZ{{\mathbb Z}}
\def\SS {\mathbb{S}}

\numberwithin{equation}{section}

\begin{document}
\title[Gevrey regularity for non-cutoff Boltzmann equation ]
{Gevrey regularity of spatially homogeneous \\ Boltzmann equation without cutoff}

\author{Teng-Fei Zhang }
\address{Teng-Fei Zhang
\newline\indent
Department of Mathematics, Sun Yat-sen University
\newline\indent
510275, Guangzhou, P. R. China}
\email{fgeyirui@163.com}

\author{Zhaoyang Yin }
\address{Zhaoyang Yin
\newline\indent
Department of Mathematics, Sun Yat-sen University
\newline\indent
510275, Guangzhou, P. R. China}
\email{mcsyzy@mail.sysu.com.cn}



\date{12-December-2011}

\begin{abstract}
In this paper, we study the Gevrey regularity of spatially homogeneous Boltzmann equation without angular cutoff. We prove the propagation of Gevrey regularity for $C^\infty$ solutions with the Maxwellian decay to the Cauchy problem of spatially homogeneous Boltzmann equation. The idea we use here is based on the framework of Morimoto's recent paper
(See Morimoto: J. Pseudo-Differ. Oper. Appl. (2010) 1: 139-159, DOI:10.1007/s11868-010-0008-z), but we extend the range of the index $\gamma$ satisfying $\gamma + 2s \in (-1,1)$, $s\in (0,1/2)$ and in this case we consider the kinetic factor in the form of $\Phi(v)=|v|^\gamma$ instead of $\la v \ra ^\gamma$ as Morimoto did before.
\end{abstract}

\maketitle

\section{Introduction}\label{Section1}
We first introduce the Boltzmann equation for the spatially inhomogeneous case:
$$
f_t(t,x,v)+v\cdot \nabla_xf(t,x,v)=Q(f,f)(v),\quad t\in \RR^+, \,\,  x,v \in \RR^3 \,,
$$
where $f= f(t,x,v)$ is the density distribution function
of particles located around position $x\in \RR^3$ with velocity $v\in \RR^3$ at time $t \geq 0$.
The right-hand side of the above equation is the so called Boltzmann bilinear collision operator acting only on the velocity variable $v$:
\[
Q(g, f)=\int_{\RR^3}\int_{\mathbb S^{2}}B\left({v-v_*},\sigma
\right)
 \left\{g'_* f'-g_*f\right\}d\sigma dv_*\,.
\]

    Hereafter we use the notation $f=f(t,x,v)$, $f_*=f(t,x,v_*) $, $f'=f(t,x,v') $, $f'_*=f(t,x,v'_*) $,
and for convenience we choose the $\sigma$-representation to describe the relations between the post- and pre-collisional velocities, that is, for $\sigma \in \mathbb S^2$,
$$
v'=\frac{v+v_*}{2}+\frac{|v-v_*|}{2}\sigma,\,\,\, v'_*
=\frac{v+v_*}{2}-\frac{|v-v_*|}{2}\sigma.\,
$$
We note that the collision process satisfies the conservation of momentum and kinetic energy, that is,
$$
v+v_*=v'+v'_* \,,\qquad  |v|^2+|v_*|^2=|v'|^2+|v'_*|^2\,.
$$

    The collision cross section $B(z, \sigma)$ is a given non-negative function, and depends only on
the interaction law between particles. By a mathematical language, that means $B(z, \sigma)$ depends only on the relative velocity $|z|=|v-v_*|$ and the deviation angle $\theta$ through the scalar product
$\cos \theta=\frac{z}{|z|}\,\cdot\, \sigma$.

    In what follows, we consider the case in which the cross section $B$ can be assumed to be of the form:
$$
B(v-v_*, \cos \theta)=\Phi (|v-v_*|) b(\cos \theta),\,\,\,\,\,
\cos \theta=\frac{v-v_*}{|v-v_*|} \, \cdot\,\sigma\, , \,\,\,
0\leq\theta\leq\frac{\pi}{2},
$$
where the kinetic factor $\Phi$ is given by
$$
\Phi (|v-v_*|) = |v-v_*|^{\gamma},
$$
and the angular part $b$ with singularity satifies,
$$
\sin \theta b(\cos \theta) \sim K\theta^{-1-2s}, \ \ \mbox{as} \ \ \theta\rightarrow 0+,
$$
for some positive constant $K$ and $0< s <1$.

    We remark here if the inter-molecule potential satisfies the inverse-power law potential
$U(\rho) = \rho ^{-(p-1)}$, $ p>2 $, it holds that $ \gamma = \frac{p-5}{p-1} $,
$ s=\frac{1}{p-1} $. Generally, the case $\gamma >0$, $\gamma =0$, and $\gamma <0$ correspond to so called hard, Maxwellian, and soft potentials respectively. And the case $0<s<1/2$, $1/2 \leq s<1$ correspond to so called mild singularity and strong singularity respectively.

    In the paper we consider the Cauchy problem for the spatially homogeneous Boltzmann equation without
cutoff, with a $T >0$,
\begin{equation}\label{BE}
\left\{
\begin{array}{l}\displaystyle
f_t(t,v)=Q(f,f)(v),\quad t\in (0,T], \,  v \in \RR^n,
\\
f(0,v)=f_0(v), \end{array}
\right.
\end{equation}
where ``spatially homogeneous" means that $f$ depends only on $t$ and $v$.

    Let us give a brief review about the study for Boltzmann equation. Since Grad introduced a cutoff
assumption for cross section $B$ due to the difficulties coming from the singularity of the angular part $b$, there have been a lot of results. And as for the non-cutoff theory, great progress has been made in recent years, for that we refer the reader to Alexandre's review paper \cite{Alex-review}. And when considering non-cutoff Boltzmann equation in Gevrey spaces (see Definition \ref{Def2} below), Ukai in \cite{ukai} shows that the Cauchy problem for the Boltzmann equation admits a unique local solution in Gevrey classes for both spatially homogeneous and inhomogeneous cases, under the assumption on the cross section:
\begin{align*}
&\big| B(|z|,\cos \theta) \big| \leq  K(1+|z|^{-\gamma'}\!\!+|z|^\gamma) \theta^{-n+1-2s}, \quad n  \textrm{ is dimensionality},\\
&(0\leq \gamma' < n,\,\, 0\leq \gamma <2,\,\, 0\leq s<1/2, \,\,\gamma +6s<2 ).
\end{align*}
By introducing the norm of Gevrey space
$$
\|f\|^{U}_{\delta,\rho,\nu} = \sum_{\alpha} \frac{\rho^{|\alpha|}}{\{\alpha!\}^\nu}
\|e^{\delta \la v\ra^2} \partial_v^{\alpha} f \|_{L^\infty(\RR^n_v)},
$$
it was proved that in the spatially homogeneous case, for instance, under some assumptions for $\nu$ and the initial datum $f_0(v)$, the Cauchy problem (\ref{BE}) has a unique solution $f(t,v)$ for $t\in (0,T]$.

    We then turn to the work of Devillettes, he firstly studied in \cite{Desv2} the $C^\infty$ smoothing
effect for solutions of Cauchy problem for spatially homogeneous non-cutoff case, and conjectured Gevrey smoothing effect. Some years later, he proved in \cite{Desv} the propagation of Gevrey regularity for solutions without any assumption on the decay at infinity in $v$ variables.

    So far there have been extensive studies on the Gevrey regularity of solutions. We remark that in
\cite{MUXY-DCDS}, Morimoto et al. considered the Gevrey regularity for the linearized Boltzmann equation around the absolute Maxwellian distribution, by using a mollifier as follows:
$$
G_\delta (t,D_v)=\frac{e^{t \la D_v \ra^{1/\nu}}}{1+\delta e^{t \la D_v \ra^{1/\nu}}},\quad
0< \delta <1.
$$
The same method was used later for many related research. We refer the reader to \cite{ultra-analytic, chenhua1, chenhua2} which were about the ultra-analytic smoothing effect for spatially homogeneous nonlinear Landau equation in the Maxwellian case and the linear and non-linear Fokker-Planck equations, and \cite{kac} which studied the Kac's equation (a simplification of Boltzmann equation to one dimension case).

    Recently Morimoto considered in \cite{mu} the Gevrey regularity of $C^\infty$ solutions with the
Maxwellian decay to the Cauchy problem of spatially homogeneous Boltzmann equation. We here consider the general kinetic factor $\Phi (|v-v_*|) = |v-v_*|^{\gamma}$ taking the place of the modified kinetic factor
$\Phi (|v-v_*|) =( 1+|v-v_*|^2)^{\gamma/2}$ used in \cite{mu}, and extend his results for the range of $\gamma$ by using some new estimates.

    In the paper we consider the mild singularity case $0<s<1/2$, and we assume that
$$-1<\gamma +2s<1.$$
In the case $0<s<1/2$, Huo et al. proved in \cite{HMUY} that any weak solution $f(t,v)$ to (\ref{BE}) satisfying the natural boundedness on mass, energy and entropy, namely,
\begin{align}\label{natural bound}
\int_{\RR^n} \!\!\!f(v)[1+|v|^2+\log(1+f(v))]dv < +\infty,
\end{align}
belongs to $H^{+\infty}(\RR^n)$ for any $0<t \leq T$, and moreover,
\begin{align}\label{smooth solution}
f \in L^\infty \big([t_0,T];H^{+\infty}(\RR^n)\big),
\end{align}
for any $T>0$ and $t_0 \in (0,T)$. And in the paper \cite{amuxy-regulariz} Alexandre et al. considered a kind of solution having the Maxwellian decay, which means that,
\begin{align}\label{Maxwellian decay}
\textrm{there exists a } \delta_0>0 \textrm{ such that }
e^{\delta_0 \la v \ra ^2} f\in L^\infty \big( [t_0,T];H^{+\infty}(\RR^n) \big).
\end{align}

    We mention that we could assume $t_0=0$ in above equation by translation when considering the Gevrey
regularity. Then we introduce the following definition:
\begin{defi}\label{Def1}
We say that $f(t,v)$ is a smooth Maxwellian decay solution to the Cauchy problem (\ref{BE}) if
\begin{align*}
\left\{
\begin{array}{l}\displaystyle
f \geq 0,\,\, \not \equiv 0,
\\
\exists \delta_0 >0 \textrm{ such that }
e^{\delta_0 \la v \ra ^2} f\in L^\infty \big( [0,T];H^{+\infty}(\RR^n) \big).
\end{array}
\right.
\end{align*}
\end{defi}
    It should be noted that the same arguments as in the proof of Theorem 1.2 of \cite{amuxy-regulariz}
shows the uniqueness of the smooth Maxwellian decay solution to the Cauchy problem (\ref{BE}).

    Then before ending our review of Boltzmann equation, we recall the definition of
Gevrey spaces $G^s(\Omega)$ where $\Omega$ is an open subset of $\RR^n$. (It could be found in \cite{ultra-analytic}.)
\begin{defi}\label{Def2}
For $0<s<+\infty$, we say that $ f \in G^s(\Omega) $, if $f \in C^\infty(\Omega)$, and there exist $C>0,\,N_0>0$ such that
$$ \|\partial^\alpha f\|_{L^2(\Omega)} \leq C^{|\alpha|+1} {\{\alpha! \}^s},\quad
\forall \alpha \in \NN^n,\,|\alpha| \geq N_0.
$$
If the boundary of $\Omega$ is smooth, by using the Sobolev embedding theorem, we have the same type estimate
with $L^2$ norm replaced by any $L^p$ norm for $2 < p \leq +\infty$. On the whole space $\Omega = \RR^n$, it is also equivalent to
$$
e^{c_0 (-\Delta)^{1/(2s)}} (\partial ^{\beta_0} f) \in L^2(\RR^n), \quad
$$
for some $c_0>0$ and  $\beta_0 \in \NN^n$, where $e^{c_0 (- \Delta)^{1/(2s)}} $ is the Fourier multiplier defined by
$$
e^{c_0 (- \Delta)^{1/(2s)}} u(x)= \mathcal{F}^{-1}\big( e^{c_0 |\xi|^{1/s}}\textrm{\^{u}}(\xi) \big)\,.
$$
If $s = 1$, it is usual analytic function. If $s > 1$, it is Gevrey class function. For $0 < s < 1$,
it is called ultra-analytic function.
\end{defi}

    In the paper, we consider the smooth solution satisfying the following assumptions:
\begin{itemize}
\item[$\bf P_1$:] The solution satisfying (\ref{natural bound}) exists and satisfies
(\ref{smooth solution}), moreover, has the same Maxwellian decay as (\ref{Maxwellian decay}).
\item[$\bf P_2$:] There is a unique smooth Maxwellian decay solution to the Cauchy problem (\ref{BE}) (similar as the case considered in \cite{amuxy-regulariz}).
\end{itemize}

    Now we give our main result in the paper:
\begin{theo}\label{Propagation of Gevrey}
Let $\nu >1 $(which is indepdent of s) and assume that $ 0<s<1/2 $, $ -1<\gamma+2s<1$. Let $f(t,v)$ be a smooth
Maxwellian decay solution to the Cauchy problem (1.1). If there exist $\rho'$, $\delta'$ such that
\begin{align}\label{initial assumption}
\sup_{\alpha} \frac{\rho'^{|\alpha|} \|e^{\delta' \la v \ra ^2}\partial^{\alpha}_{v} f(0)\|_{L^2}}{\{\alpha!\}^\nu} < +\infty\,,
\end{align}
then there exist $ \rho >0 $ and $ \delta,\kappa >0 $ with $ \delta >\kappa T $ such that

\begin{align}\label{propag}
\sup_{t\in (0,T]}\sup_{\alpha} \frac{\rho^{|\alpha|} \|e^{(\delta-\kappa t) \la v \ra ^2}\partial^{\alpha}_{v} f(t)\|_{L^2}}{\{\alpha!\}^\nu} < +\infty\,.
\end{align}
\end{theo}

\begin{rema}
It should be noted that the above theorem is similar as Theorem 1.2 in \cite{mu}, but we here extend the range of $\gamma$ and consider $\Phi=|v|^\gamma$.
\end{rema}

    By using the similar arguments as Section 4 in \cite{mu}, we obtain the Gevrey smoothing effect
of order $1/(2s) $ as follows:
\begin{theo}\label{Gevrey Regularity}
Assume that $ 0<s<1/2 $, $ -1<\gamma+2s<1$. Let $ \nu =1/(2s)$ and let $f(t,v)$ be a smooth
Maxwellian decay solution to the Cauchy problem (\ref{BE}), then for any $ t_0 \in (0,T)$, there exist $ \rho >0 $ and $ \delta,\kappa >0 $ with $ \delta >\kappa T $ such that
\begin{align}
\sup_{t\in [t_0,T]}\sup_{\alpha} \frac{\rho^{|\alpha|} \|e^{(\delta-\kappa t) \la v \ra ^2}\partial^{\alpha}_{v} f(t)\|_{L^2}}{\{\alpha!\}^\nu} < +\infty\,.
\end{align}
\end{theo}

    The rest of the paper will be organized as follows. In Section 2, we will cite some
definitions and main lemma as the authors of \cite{mu} did, then we could complete immediately
the proof of Theorem \ref{Propagation of Gevrey}. The proof of one main lemma will be given
in Section 3, where we use some new estimates on the collision operator in the framework of \cite{mu},
and obtain the new result improving the range of $\gamma$.


\section{Proof of Theorem \ref{Propagation of Gevrey}}\label{Section2}
\smallskip
    Firstly we introduce some basic definitions (see \cite{mu} for details).\\
Let $l,\,r \in \ZZ_+$ which will be fixed later. For $\delta,\, \rho >0$ we set:
$$
\|f\|_{\delta,l,\rho,\alpha,r}=\frac{ \rho^{|\alpha|} \| \la v \ra ^l e^{\delta \la v \ra ^2} \partial^\alpha_v f\|_{L^2}}{\{(\alpha - r)!\}^\nu}\,,
$$
where $ \alpha =(\alpha_1,\alpha_2,\ldots,\alpha_n) \in \mathbb{Z}^n_+ $, and we denote
$$ (\alpha - r)!=(\alpha_1 - r)! \cdots (\alpha_n - r)! \,.$$

    Now we give the following defintion:
\begin{align}
\|f\|_{l,\rho,r,N}(t) = \sup_{rn \leq |\alpha| \leq N} \|f\|_{\delta -\kappa t,l,\rho,\alpha,r}\,,
\end{align}
with fixed $\delta,\,\kappa >0$ satisfying $\delta > \kappa T$. Here $N$ satisfying $rn \leq |\alpha| \leq N$ is a fixed large number.
Then for $ h > 1$ we could obtain
\begin{align}
\|f\|_{l,\rho (1+h)^{-\nu},r,N}(t) \leq \Big( \frac{(r!)^n}{h^r}\Big)^\nu \|f\|_{l,\rho,0,N}(t)\,.
\end{align}

Now let $\rho =\rho'$ in the above inequality and take a large enough $ h $. Then it follows from the initial assumption (\ref{initial assumption}) that $ \|f\|_{l,\rho' (1+h)^{-\nu},r,N}(0) $ is as small as we want, where $\delta$ can be chosen any positive less than $\delta' > 0$ in (\ref{initial assumption}).

So (\ref{propag}) could be proved if only we could prove that (with $\rho =\rho' (1+h)^{-\nu}$)
\begin{align}
\sup_{t \in (0,T]} \|f\|_{l,\rho,r,N}(t) < \infty\,,
\end{align}
under the assumption that $ \|f\|_{l,\rho,r,N}(0) $ is sufficiently small.

    We mention that we will consider the Cauchy problem in $\RR^3$ in the paper.
\begin{lemm}\label{Main Lemma}
If $ l \geq 4$ and $r >1+\nu/(\nu-1)$ then for any $\alpha$ satisfying $ 3r \leq |\alpha| \leq N $ we have
\begin{align}\label{basic inequa}
&\|f(t)\|^2_{\delta -\kappa t,l,\rho,\alpha,r}
+ 2\kappa \!\int^t_0 \! \|f(\tau)\|^2_{\delta-\kappa t,l+1,\rho,\alpha,r}d\tau \\
\leq
&\|f(0)\|^2_{\delta,l,\rho,\alpha,r}
       + C_\kappa \!\int^t_0\! \Big( \|f\|^2_{l,\rho,r,N}(\tau)
                             + \|f\|^{2(1+\beta)/\beta}_{l,\rho,r,N}(\tau) \Big)d\tau \nonumber\\
&+ \frac{\kappa}{10}\sup_{3r \leq |\alpha| \leq N}
                         \int^t_0\! \|f(\tau)\|^2_{\delta-\kappa t,l+1,\rho,\alpha,r}d\tau, \nonumber
\end{align}
where $\beta = 1-(\gamma+2s)$.
\end{lemm}

    Then we could prove Theorem \ref{Propagation of Gevrey} with the same arguments as in Section 2
in \cite{mu}, and we omit the proof here. The proof of this lemma will be given in the next section.


\section{Proof of Lemma \ref{Main Lemma}}\label{Section3}
\smallskip
    Let $\mu =\mu_{\delta,\kappa}(t) =e^{-(\delta-\kappa t) \la v \ra ^2}  $ with $\delta > \kappa T$.
Since the translation invariance of the collision operator with respect to the variable $v$ implies that (see \cite{Desv2,ukai}), for the translation operator $\tau_h $ in $v$ by $h$, we have
$$
\tau_h Q(f,g)=Q(\tau_h f,\tau_h g).
$$
Thus we have
$$
\partial_v^\alpha Q(f,g) = \sum_{\alpha=\alpha'+\alpha''}
                           \frac{\alpha!}{\alpha'! \alpha''!} Q\big(f^{(\alpha')},f^{(\alpha'')}\big).
$$

Then we could obtain from Eq.(\ref{BE}) that
$$
\partial_t (\partial_v^\alpha f) = Q(f,f^{(\alpha)})+\sum_{\alpha' \neq 0}
                            \frac{\alpha!}{\alpha'! \alpha''!}Q\big(f^{(\alpha')},g^{(\alpha'')}\big).
$$
Multiplying both sides of the above equation by $\mu^{-1}$, we obtain
\begin{align}\label{eq1}
(\partial_t +\kappa \la v \ra ^2)(\mu^{-1} \partial_v^\alpha f)= \mu^{-1}Q(f,f^{(\alpha)})
 +\sum_{\alpha' \neq 0}\frac{\alpha!}{\alpha'! \alpha''!}\mu^{-1} Q\big(f^{(\alpha')},f^{(\alpha'')}\big).
\end{align}
Set $F=\mu^{-1}f$ and denote $F^{(\alpha)}=\mu^{-1}f^{(\alpha)}$ for $\alpha \in \ZZ^n_+$. Noticing that $\mu \mu_* = \mu' \mu'_*$, we get the following formula
$$
\mu^{-1}Q(f,g)=Q(\mu F,G) + \iint\! B(\mu_*-\mu'_*)F'_* G'dv_* d\sigma.
$$
Then it follows from (\ref{eq1}) that
\begin{align*}
(\partial_t +\kappa \la v \ra ^2) F^{(\alpha)}=&Q(\mu F,F^{(\alpha)})
+\sum_{\alpha' \neq 0}\frac{\alpha!}{\alpha'! \alpha''!}Q\big(\mu F^{(\alpha')},F^{(\alpha'')}\big)\\
&+\sum_{\alpha=\alpha'+\alpha''}\frac{\alpha!}{\alpha'! \alpha''!}
                       \iint\! B(\mu_*-\mu'_*)(F^{(\alpha')})'_* (F^{(\alpha'')})'dv_* d\sigma.
\end{align*}
Hereafter we denote $W_l=\la v \ra ^l$. Multiplying by $W_l^2F^{(\alpha)}$ both sides of the above equation and integrating with respect to $v$, we have
\begin{align}\label{EQ}
&\frac{1}{2}\frac{d}{dt}\|W_l F^{(\alpha)}\|^2 + \kappa \|W_{l+1} F^{(\alpha)}\|^2\\
=&\big( Q(\mu F,F^{(\alpha)}), W_l^2F^{(\alpha)}\big)
   +\sum_{\alpha' \neq 0}\frac{\alpha!}{\alpha'!\alpha''!}
    \Big(Q\big(\mu F^{(\alpha')},F^{(\alpha'')}\big),W_l^2F^{(\alpha)}\Big)\nonumber\\
&+\sum_{\alpha=\alpha'+\alpha''}\frac{\alpha!}{\alpha'! \alpha''!}
     \iiint\! B(\mu_*-\mu'_*)(F^{(\alpha')})'_* (F^{(\alpha'')})'W_l^2F^{(\alpha)}dv dv_* d\sigma \nonumber\\
=&\Psi_1^{(0,\alpha)}(t)
  +\sum_{\alpha' \neq 0}\frac{\alpha!}{\alpha'! \alpha''!}\Psi_1^{(\alpha',\alpha'')}(t)
  +\sum_{\alpha=\alpha'+\alpha''}\frac{\alpha!}{\alpha'! \alpha''!}\Psi_2^{(\alpha',\alpha'')}(t)
  \nonumber\\
=&\Psi_1^{(0,\alpha)}(t)+\mathcal{J}^{\alpha}(t)+\mathcal{K}^{\alpha}(t)\nonumber.
\end{align}
Then multiplying by $\frac{\rho^{2|\alpha|}}{\{(\alpha-r)!\}^{2\nu}}$ both sides of the above equation, and integrating from $0$ to $t\in (0,T]$, we obtain
\begin{align}\label{EQ2}
&\|f(t)\|^2_{\delta -\kappa t,l,\rho,\alpha,r}
+ 2\kappa \!\int^t_0 \! \|f(\tau)\|^2_{\delta-\kappa t,l+1,\rho,\alpha,r}d\tau \\
\leq
&\|f(0)\|^2_{\delta,l,\rho,\alpha,r}
 +2\int^t_0\! \frac{\rho^{2|\alpha|}}{\{(\alpha-r)!\}^{2\nu}}
   \big( \Psi_1^{(0,\alpha)}(\tau)+\mathcal{J}^{\alpha}(\tau)+\mathcal{K}^{\alpha}(\tau) \big)d\tau
   \nonumber.
\end{align}
Firstly we consider the estimate on $\Psi_2^{(\alpha',\alpha'')}(t)$. \\
By (\ref{EQ}), we have
\begin{align*}
&\Psi_2^{(\alpha',\alpha'')}(t)\\
=&\iiint\!B(\mu_*-\mu'_*)(F^{(\alpha')})'_* (W_{l-1}F^{(\alpha'')})'W_{l+1}F^{(\alpha)}dv dv_* d\sigma\\
&+\iiint\!B(\mu_*-\mu'_*)(F^{(\alpha')})'_* \big(W_{l-1}-W'_{l-1}\big)
  (F^{(\alpha'')})'W_{l+1}F^{(\alpha)}dv dv_* d\sigma\\
=&\Psi_{2,1}^{(\alpha',\alpha'')}(t)+\Psi_{2,2}^{(\alpha',\alpha'')}(t).
\end{align*}
Notice that, see Lemma 2.3 in \cite{amuxy-regulariz},
\begin{align}
\big|W_{l-1}-W'_{l-1}\big|
&\leq C\sin(\theta/2)\big( W'_{l-1}+W'_{1,*}W'_{l-2}
  + \sin^{(l-2)}(\theta/2)W'_{l-1,*}\big)\\
&\leq C\big( \theta W'_{l-1}W'_{1,*}+\theta^{l-1}W'_{l-1,*} \big) \nonumber.
\end{align}
We split $\Psi_{2,2}^{(\alpha',\alpha'')}(t)$ into $G_1+G_2$ corresponding to the two terms of the right-hand side. On the other hand, we could get
\begin{align}
|\mu_*-\mu'_*| \leq C\theta^\lambda |v'-v'_*|^\lambda = C\theta^\lambda |v-v_*|^\lambda,\quad
\lambda \in [0,1],\,t\in [0,T].
\end{align}
In the case of $0<\gamma +2s<1$, there exists a $\lambda \in (0,1)$ such that $\lambda>2s$, $\gamma+\lambda \leq 1$. So we have $\gamma+\lambda > 0 $ and
$|v'-v'_*|^{\gamma+\lambda}\leq \la v' \ra^{\gamma+\lambda} \la v'_* \ra^{\gamma+\lambda} $ immediately. Then we have
\begin{align*}
G_1\leq &\iiint b(\cos \theta) \theta^{\lambda +1} |v'-v'_*|^{\gamma+\lambda}
W'_{l-1}W'_{1,*}\big|(F^{(\alpha')})'_* \big|
\big|(F^{(\alpha'')})'\big|\big|(W_{l+1}F^{(\alpha)})\big|dvdv_*d\sigma\\
\leq &\iiint b(\cos \theta) \theta^{\lambda +1}
\big|(W_{1+\gamma+\lambda}F^{(\alpha')})'_* \big| \big|(W_{l-1+\gamma+\lambda}F^{(\alpha'')})'\big|\big|(W_{l+1}F^{(\alpha)})\big|dvdv_*d\sigma\\
=&\iiint b(\cos \theta) \theta^{\lambda +1}
\big|(W_{1+\gamma+\lambda}F^{(\alpha')})_* \big| \big|(W_{l-1+\gamma+\lambda}F^{(\alpha'')})\big|\big|(W_{l+1}F^{(\alpha)})'\big|dvdv_*d\sigma\\
\leq &\Big( \iiint b(\cos \theta) \theta^{\lambda +1}
\big|(W_{1+\gamma+\lambda}F^{(\alpha')})_*\big| \big|(W_{l-1+\gamma+\lambda}F^{(\alpha'')})\big|^2dvdv_*d\sigma \Big)^{1/2}\\
 &\times \Big( \iiint b(\cos \theta) \theta^{\lambda +1}
\big|(W_{1+\gamma+\lambda}F^{(\alpha')})_*\big| \big|(W_{l+1}F^{(\alpha)})'\big|^2dvdv_*d\sigma \Big)^{1/2}\\
=&G^{1/2}_{1,1}\times G^{1/2}_{1,2}.
\end{align*}
Noticing $0<s<1/2$, we have $\int_{\SS^2}\! b(\cos \theta) \theta^{\lambda+1} d\sigma  \leq C$. Thus
$$
G_{1,1}\leq C\|W_{1+\gamma+\lambda}F^{(\alpha')} \|_{L^1}\|W_{l-1+\gamma+\lambda}F^{(\alpha'')} \|^2_{L^2}.
$$
By using the change of variables $v\mapsto v'=\frac{v+v_*}{2}+\frac{|v-v_*|}{2}\sigma$ for fixed $\sigma$ and $v_*$ whose Jacobian satisfying:
$$
\left|\frac{dv'}{dv}\right|=\frac{\cos^2(\theta/2)}{4},
$$
we have then
\begin{align*}
G_{1,2}\leq &\iiint b(\cos \theta) \theta^{\lambda +1}\frac{4}{\cos^2(\theta/2)}
\big|(W_{1+\gamma+\lambda}F^{(\alpha')})_*\big| \big|(W_{l+1}F^{(\alpha)})\big|^2dvdv_*d\sigma\\
\leq &C\|W_{1+\gamma+\lambda}F^{(\alpha')} \|_{L^1}\|W_{l+1}F^{(\alpha)}\|^2_{L^2}.
\end{align*}
So we could obtain in the case of $0<\gamma +2s<1$:
\begin{align*}
G_1\leq &C\|W_{1+\gamma+\lambda}F^{(\alpha')} \|_{L^1}\|W_{l-1+\gamma+\lambda}F^{(\alpha'')} \|_{L^2}
          \|W_{l+1}F^{(\alpha)}\|_{L^2}\\
\leq &C\|W_l F^{(\alpha')} \|_{L^2} \|W_l F^{(\alpha'')} \|_{L^2} \|W_{l+1}F^{(\alpha)}\|_{L^2},
\end{align*}
if $l \geq 4>5/2+\gamma +\lambda$ by using the embedding
$$
L^2_{3/2+\varepsilon}(\RR^3) \subset L^1(\RR^3),\, \varepsilon >0.
$$

    And on the other hand,
\begin{align*}
G_2\leq &\iiint b(\cos \theta) \theta^{\lambda +l-1} |v'-v'_*|^{\gamma+\lambda}
\big|(W_{l-1}F^{(\alpha')})'_* \big|
\big|(F^{(\alpha'')})'\big|\big|(W_{l+1}F^{(\alpha)})\big|dvdv_*d\sigma\\
\leq &\iiint b(\cos \theta) \theta^{\lambda +l-1}
\big|(W_{l-1+\gamma+\lambda}F^{(\alpha')})_* \big| \big|(W_{\gamma+\lambda}F^{(\alpha'')})\big|\big|(W_{l+1}F^{(\alpha)})'\big|dvdv_*d\sigma\\
\leq &\Big( \iiint b(\cos \theta) \theta^{\lambda}
\big|(W_{\gamma+\lambda}F^{(\alpha'')})\big| \big|(W_{l-1+\gamma+\lambda}F^{(\alpha')})_*\big|^2dvdv_*d\sigma \Big)^{1/2}\\
 &\times \Big( \iiint b(\cos \theta) \theta^{\lambda +2l-2}
\big|(W_{\gamma+\lambda}F^{(\alpha'')})\big| \big|(W_{l+1}F^{(\alpha)})'\big|^2dvdv_*d\sigma \Big)^{1/2}\\
=&G^{1/2}_{2,1}\times G^{1/2}_{2,2}.
\end{align*}
Firstly we have
$$
G_{2,1}\leq C\|W_{\gamma+\lambda}F^{(\alpha'')} \|_{L^1}\|W_{l-1+\gamma+\lambda}F^{(\alpha')} \|^2_{L^2}.
$$
As for $G_{2,2}$, we use the change of variables $v_*\mapsto v'=\frac{v+v_*}{2}+\frac{|v-v_*|}{2}\sigma$
whose Jacobian satisfying:
$$
\left|\frac{dv'}{dv_*}\right|=\frac{\sin^2(\theta/2)}{4}.
$$
Then we have
\begin{align*}
G_{2,2}\leq &\iiint b(\cos \theta) \theta^{\lambda +2l-2}\frac{4}{\sin^2(\theta/2)}
\big|(W_{\gamma+\lambda}F^{(\alpha'')}) \big| \big|(W_{l+1}F^{(\alpha)})'\big|^2dvdv'd\sigma\\
\leq &C\|W_{\gamma+\lambda}F^{(\alpha'')} \|_{L^1}\|W_{l+1}F^{(\alpha)}\|^2_{L^2},
\end{align*}
if $l\geq 4 >2+s-\lambda/2$. Hence we obtain
\begin{align*}
G_2\leq &C\|W_{\gamma+\lambda}F^{(\alpha'')} \|_{L^1}\|W_{l-1+\gamma+\lambda}F^{(\alpha')} \|_{L^2}
          \|W_{l+1}F^{(\alpha)}\|_{L^2}\\
\leq &C\|W_l F^{(\alpha')} \|_{L^2} \|W_l F^{(\alpha'')} \|_{L^2} \|W_{l+1}F^{(\alpha)}\|_{L^2}.
\end{align*}

    On the other hand if $-1<\gamma +2s \leq 0$, we choose $\lambda =2s$. Noticing that
$$
\la v-v_* \ra^r \leq C \la v \ra^{|r|} \la v_* \ra^r, r \in \RR,
$$
then we have
\begin{align*}
G_1\leq &\iiint b(\cos \theta) \theta^{2s +1} \frac{|v'-v'_*|^{\gamma+2s}}{\la v'-v'_* \ra^{\gamma+2s}}
\la v'_* \ra^{-(\gamma+2s)} \la v' \ra^{\gamma+2s} W'_{l-1}W'_{1,*}
\big|(F^{(\alpha')})'_* \big| \\
&\times  \big|(F^{(\alpha'')})'\big| \big|(W_{l+1}F^{(\alpha)})\big|dvdv_*d\sigma\\
\leq &\iiint b(\cos \theta) \theta^{2s +1}\big( 1+|v'-v'_*|^{\gamma +2s}\Large\textbf{1}_{|v'-v'_*| \leq 1} \big) \big|(W_{1-\gamma-2s}F^{(\alpha')})'_* \big| \\
&\times \big|(W_{l-1+\gamma+2s}F^{(\alpha'')})'\big| \big|(W_{l+1}F^{(\alpha)})\big|dvdv_*d\sigma\\
=&\iiint b(\cos \theta) \theta^{2s +1} \big( 1+|v-v_*|^{\gamma +2s}\Large\textbf{1}_{|v-v_*| \leq 1} \big) \big|(W_{1-\gamma-2s}F^{(\alpha')})_* \big|\\
&\times \big|(W_{l-1+\gamma+2s}F^{(\alpha'')})\big|\big|(W_{l+1}F^{(\alpha)})'\big|dvdv_*d\sigma\\
\leq &\Big( \iiint b(\cos \theta) \theta^{2s +1}\big( 1+|v-v_*|^{\gamma +2s}\Large\textbf{1}_{|v-v_*| \leq 1} \big) \big|(W_{1-\gamma-2s}F^{(\alpha')})_*\big|\\
&\qquad \times \big|(W_{l-1+\gamma+2s}F^{(\alpha'')})\big|^2dvdv_*d\sigma \Big)^{1/2}\\
 &\times \Big( \iiint b(\cos \theta) \theta^{2s +1}\big( 1+|v-v_*|^{\gamma +2s}\Large\textbf{1}_{|v-v_*| \leq 1} \big) \big|(W_{1-\gamma-2s}F^{(\alpha')})_*\big| \\
&\qquad \times \big|(W_{l+1}F^{(\alpha)})'\big|^2dvdv_*d\sigma \Big)^{1/2}\\
=&G^{1/2}_{1,3}\times G^{1/2}_{1,4}.
\end{align*}
We mention that there exists a positive $\delta$ such that $\gamma +2s=\delta -1$, which implies that
$$
\chi(v)=|v|^{\gamma +2s}\Large\textbf{1}_{|v| \leq 1} \in L^{\frac{3-\varepsilon}{1-\delta}}.
$$
Choosing $\varepsilon=3\delta$, we have $\chi(v)\in L^3$. And since $0<s<1/2$, we have
$$
G_{1,3}\leq C
\big(\|W_{1-\gamma-2s}F^{(\alpha')} \|_{L^1}+ \|W_{1-\gamma-2s}F^{(\alpha')} \|_{L^{\frac{3}{2}}}\big) \|W_{l-1+\gamma+2s}F^{(\alpha'')} \|^2_{L^2}.
$$
By using the change of variables $v\mapsto v'$, we get
\begin{align*}
G_{1,4}\leq &\iiint b(\cos \theta) \theta^{2s +1}\frac{4}{\cos^2(\theta/2)}
\big( 1+|v-v_*|^{\gamma +2s}\Large\textbf{1}_{|v-v_*| \leq 1} \big)\\
&\times \big|(W_{1-\gamma-2s}F^{(\alpha')})_*\big| \big|(W_{l+1}F^{(\alpha)})\big|^2dvdv_*d\sigma\\
\leq &C\big(\|W_{1-\gamma-2s}F^{(\alpha')} \|_{L^1}
+ \|W_{1-\gamma-2s}F^{(\alpha')} \|_{L^{\frac{3}{2}}}\big)
\|W_{l+1}F^{(\alpha)}\|^2_{L^2}.
\end{align*}
So we obtain in the case of $-1<\gamma+2s\leq 0$:
\begin{align*}
G_1\leq &C
\big(\|W_{1-\gamma-2s}F^{(\alpha')} \|_{L^1} + \|W_{1-\gamma-2s}F^{(\alpha')} \|_{L^{\frac{3}{2}}}\big)
\|W_{l-1+\gamma+2s}F^{(\alpha'')}\|_{L^2} \\
&\times \|W_{l+1}F^{(\alpha)}\|_{L^2}\\
\leq &C\|W_l F^{(\alpha')} \|_{L^2} \|W_l F^{(\alpha'')} \|_{L^2} \|W_{l+1}F^{(\alpha)}\|_{L^2},
\end{align*}
if $l \geq 4>\max{\{5/2-(\gamma +2s), 3/2-(\gamma +2s)\}}$ by using the embedding
$$
L^2_{3/2+\varepsilon}(\RR^3) \subset L^1(\RR^3),\quad L^2_{1/2+\varepsilon}(\RR^3) \subset L^{\frac{3}{2}}(\RR^3),\,\varepsilon >0.
$$

    And as for $G_2$, we have
\begin{align*}
G_2\leq &\iiint b(\cos \theta) \theta^{l-1+2s} |v'-v'_*|^{\gamma+2s}
\big|(W_{l-1}F^{(\alpha')})'_* \big|
\big|(F^{(\alpha'')})'\big|\big|(W_{l+1}F^{(\alpha)})\big|dvdv_*d\sigma\\
\leq &\iiint b(\cos \theta) \theta^{l-1+2s}\big( 1+|v-v_*|^{\gamma +2s}\Large\textbf{1}_{|v-v_*| \leq 1} \big) \big|(W_{l-1+\gamma+2s}F^{(\alpha')})_* \big| \\
&\times \big|(W_{-\gamma-2s}F^{(\alpha'')})\big| \big|(W_{l+1}F^{(\alpha)})'\big|dvdv_*d\sigma\\
\leq &\Big( \iiint b(\cos \theta)\, \theta \, \big( 1+|v-v_*|^{\gamma +2s}\Large\textbf{1}_{|v-v_*| \leq 1} \big) \,\big|(W_{-\gamma-2s}F^{(\alpha'')})\big| \\
&\qquad \times \big|(W_{l-1+\gamma+2s}F^{(\alpha')})_*\big|^2dvdv_*d\sigma \Big)^{1/2}\\
 &\times \Big( \iiint b(\cos \theta) \theta^{2l-3+4s}\big( 1+|v-v_*|^{\gamma +2s}\Large\textbf{1}_{|v-v_*| \leq 1} \big) \big|(W_{-\gamma-2s}F^{(\alpha'')})\big| \\
&\qquad \times\big|(W_{l+1}F^{(\alpha)})'\big|^2dvdv_*d\sigma \Big)^{1/2}\\
=&G^{1/2}_{2,3}\times G^{1/2}_{2,4}.
\end{align*}
Then we have the following estimate:
$$
G_{2,3}\leq C\big(\|W_{-\gamma-2s}F^{(\alpha'')} \|_{L^1} + \|W_{-\gamma-2s}F^{(\alpha'')} \|_{L^{\frac{3}{2}}} \big) \|W_{l-1+\gamma+2s}F^{(\alpha')} \|^2_{L^2}.
$$
By using the change of variables $v_*\mapsto v'$, we get
\begin{align*}
G_{2,4}\leq &\iiint b(\cos \theta) \theta^{2l-3+4s}\frac{4}{\sin^2(\theta/2)}
\big|(W_{-\gamma-2s}F^{(\alpha'')}) \big| \big|(W_{l+1}F^{(\alpha)})'\big|^2dvdv'd\sigma\\
\leq &C\big(\|W_{-\gamma-2s}F^{(\alpha'')} \|_{L^1} + \|W_{-\gamma-2s}F^{(\alpha'')} \|_{L^{\frac{3}{2}}} \big) \|W_{l+1}F^{(\alpha)}\|^2_{L^2},
\end{align*}
if $l>5/2-s$. Hence we obtain
\begin{align*}
G_2\leq &C
\big(\|W_{-\gamma-2s}F^{(\alpha'')} \|_{L^1} + \|W_{-\gamma-2s}F^{(\alpha'')} \|_{L^{\frac{3}{2}}} \big) \|W_{l-1+\gamma+2s}F^{(\alpha')} \|_{L^2} \\
&\times \|W_{l+1}F^{(\alpha)}\|_{L^2}\\
\leq &C\|W_l F^{(\alpha')} \|_{L^2} \|W_l F^{(\alpha'')} \|_{L^2} \|W_{l+1}F^{(\alpha)}\|_{L^2}.
\end{align*}
if $l \geq 4>\max{\{3/2-(\gamma +2s), 1/2-(\gamma +2s),5/2-s\}}$.

    So we have the estimate for $\Psi_{2,2}^{(\alpha',\alpha'')}$ with $-1<\gamma+2s<1$:
\begin{align}\label{Psi_22}
\left|\Psi_{2,2}^{(\alpha',\alpha'')}\right| \leq G_1+G_2
\leq C\|W_l F^{(\alpha')} \|_{L^2} \|W_l F^{(\alpha'')} \|_{L^2} \|W_{l+1}F^{(\alpha)}\|_{L^2},
\end{align}
if $l\geq 4$.

    As for $\Psi_{2,1}^{(\alpha',\alpha'')}$, we choose the $\lambda \in (0,1)$ such that
$\lambda >2s,\, \gamma+\lambda \leq 1$ and hence $1\geq \gamma+\lambda>-1$. If $0<\gamma+\lambda \leq 1$, then we have
\begin{align*}
\left|\Psi_{2,1}^{(\alpha',\alpha'')}\right|
\leq &C\iiint\!b(\cos \theta)\theta^\lambda|v'-v'_*|^{\gamma +\lambda}\big|(F^{(\alpha')})'_*\big| \big|(W_{l-1}F^{(\alpha'')})'\big|\big|(W_{l+1}F^{(\alpha)})\big|dv dv_* d\sigma \\
\leq &C\iiint\!b(\cos \theta)\theta^\lambda\big|(W_{\gamma +\lambda}F^{(\alpha')})_* \big|
\big|(W_{l-1+\gamma +\lambda}F^{(\alpha'')})\big|\big|(W_{l+1}F^{(\alpha)})'\big|dv dv_* d\sigma
\\
\leq &C\|W_{\gamma+\lambda}F^{(\alpha')} \|_{L^1}\|W_{l-1+\gamma+\lambda}F^{(\alpha'')} \|_{L^2}
          \|W_{l+1}F^{(\alpha)}\|_{L^2} .
\end{align*}
And if $-1<\gamma+\lambda \leq 0$, we could obtain the similar estimate as on $G_1$ :
\begin{align*}
\left|\Psi_{2,1}^{(\alpha',\alpha'')}\right|
\leq &C\iiint\!b(\cos \theta)\theta^\lambda \big( 1+|v'-v'_*|^{\gamma +\lambda}\Large\textbf{1}_{|v'-v'_*| \leq 1} \big)
\big|(W_{-\gamma-\lambda} F^{(\alpha')})'_*\big| \\
&\times \big|(W_{l-1+\gamma+\lambda} F^{(\alpha'')})'\big|\big|(W_{l+1}F^{(\alpha)})\big|dv dv_* d\sigma \\
= &C\iiint\!b(\cos \theta)\theta^\lambda \big( 1+|v-v_*|^{\gamma +\lambda}\Large\textbf{1}_{|v-v_*| \leq 1} \big) \big|(W_{-\gamma -\lambda}F^{(\alpha')})_* \big|\\
&\times \big|(W_{l-1+\gamma +\lambda}F^{(\alpha'')})\big|\big|(W_{l+1}F^{(\alpha)})'\big|dv dv_* d\sigma
\\
\leq &C\big( \|(W_{-\gamma -\lambda}F^{(\alpha')}) \|_{L^1} + \|(W_{-\gamma -\lambda}F^{(\alpha')}) \|_{L^{\frac{3}{2}}} \big)\|W_{l-1+\gamma+\lambda}F^{(\alpha'')} \|_{L^2}\\
&\times \|W_{l+1}F^{(\alpha)}\|_{L^2} .
\end{align*}
So we have under the assumption $-1<\gamma+2s<1$:
\begin{align}\label{Psi_21}
\left|\Psi_{2,1}^{(\alpha',\alpha'')}\right|
\leq C\|W_l F^{(\alpha')} \|_{L^2} \|W_l F^{(\alpha'')} \|_{L^2} \|W_{l+1}F^{(\alpha)}\|_{L^2},
\end{align}
if $l \geq 4$.

    Since $\Psi_{2,2}^{(\alpha',\alpha'')}$ and $\Psi_{2,1}^{(\alpha',\alpha'')}$ have the same
bound \big(see (\ref{Psi_22}) and (\ref{Psi_21})\big), we obtain
\begin{align}
\frac{\rho^{2|\alpha|}\big| \Psi_2^{(\alpha',\alpha'')}(t)\big|}{\{(\alpha-r)!\}^{2\nu}}
\leq & C \frac{\{(\alpha'-r)!\}^\nu\{(\alpha''-r)!\}^\nu}{\{(\alpha-r)!\}^\nu}\\
&\times \|f(t)\|_{\delta-\kappa t,l,\rho,\alpha',r}
       \|f(t)\|_{\delta-\kappa t,l,\rho,\alpha'',r}\|f(t)\|_{\delta-\kappa t,l+1,\rho,\alpha,r}.
       \nonumber
\end{align}

    As for $\Psi_1^{(0,\alpha)}$ we decompose
\begin{align*}
\Psi_1^{(0,\alpha)}=&\left( Q(\mu F,W_l F^{(\alpha)}),W_l F^{(\alpha)} \right)\\
                &+\left( W_l Q(\mu F,F^{(\alpha)})-Q(\mu F,W_l F^{(\alpha)}),W_l F^{(\alpha)} \right)\\
                   =&\Psi_{1,1}^{(0,\alpha)}+\Psi_{1,2}^{(0,\alpha)}.
\end{align*}

In order to estimate the first term $\Psi_{1,1}^{(0,\alpha)}$ we cite the following coercivity estimate given in \cite{chenyemin} (we mention that the form here is slight different from \cite{chenyemin} because it need only to consider $0<s<1/2$).
\begin{lemm}\label{coercivity}
Let $0<s<1/2$ and assume that the nonnegative function $g$ satisfies
$$
\|g\|_{L^1_2(\RR^3_v)}+\|g\|_{L\log L(\RR^3_v)}<\infty.
$$
Then there exists a constant $C_g>0$ depending on $B$, $\|g\|_{L^1_1(\RR^3_v)}$ and
$\|g\|_{L\log L(\RR^3_v)}$ such that in the case of $0<\gamma +2s<1$, there holds
\begin{align}
-\left(Q(g,f),f\right)_{L^2}\geq
&C_g\|f\|^2_{H^s_{\gamma/2}} - C\|g\|_{L^1_{|\gamma|}}\|f\|^2_{H^\eta_{\gamma/2}}\\
&-C\big( \|g\|_{L^1_{\widetilde{\gamma}}}
+ C_g\|g\|^2_{L^1} \big) \|f\|^2_{L^2_{\gamma/2}}\,, \nonumber
\end{align}
where $\eta <s$ depends on $\gamma$, $s$ and
$\widetilde{\gamma}=|\gamma+2|\Large\emph{1}_{\gamma \leq 0}+|\gamma-2|\Large\emph{1}_{\gamma \geq 0};$
\\
And in the case of $-1<\gamma +2s \leq 0$, there holds
\begin{align}
-\left(Q(g,f),f\right)_{L^2}\geq
&C_g\|f\|^2_{H^s_{\gamma/2}} -C\big( \|g\|_{L^1_{|\gamma|}}
+ \|g\|_{L^{\frac{3}{2}}_{|\gamma|}} \big) \|f\|^2_{H^\eta_{\gamma/2}}\\
&- C\|g\|_{L^1_{|\gamma+2|}}\|f\|^2_{L^2_{\gamma/2}}\,, \nonumber
\end{align}
with $\eta <s$.
\end{lemm}
We omit the proof here. Using this lemma with $g=\mu F$, $f=W_l F^{(\alpha)}$, then in the case of $0<\gamma+2s<1$ we have
\begin{align}
\Psi_{1,1}^{(0,\alpha)}+c_0\| W_{l+\gamma/2} F^{(\alpha)} \|^2_{H^s} \leq
&C\big( \| \mu W_{\widetilde{\gamma}} F \|_{L^1}
      +C\| \mu F \|^2_{L^1} \big)\|W_{l+s/2} F^{(\alpha)}\|^2_{L^2}\\
&+ C\| \mu W_{|\gamma|} F \|_{L^1} \| W_{l+\gamma/2} F^{(\alpha)} \|^2_{H^\eta}, \nonumber
\end{align}
where $c_0$ is a constant depending only on the bounds of $\|f\|_{L^1_1}$, $ \|f\|_{L\log L}$.\\
We now need the following interpolation inequality, for $0<\eta <s$ and any $\varepsilon >0$,
$$
\|f\|_{H^\eta}\leq \|f\|^{\eta/s}_{H^s} \|f\|^{(s-\eta)/s}_{L^2}
              \leq \varepsilon \|f\|_{H^s} + \varepsilon^{-\frac{\eta}{s-\eta}} \|f\|_{L^2}.
$$
Then
\begin{align*}
&\| \mu W_{|\gamma|} F \|_{L^1} \| W_{l+\gamma/2} F^{(\alpha)} \|^2_{H^\eta} \\
\leq &\varepsilon \| W_{l+\gamma/2} F^{(\alpha)} \|^2_{H^s}
+C_\varepsilon \| \mu W_{|\gamma|} F \|^{s/(s-\eta)}_{L^1} \| W_{l+\gamma/2} F^{(\alpha)} \|^2_{L^2}.
\end{align*}
Now choosing $\varepsilon=c_0/2$, we have
\begin{align}\label{coercive-1}
&\Psi_{1,1}^{(0,\alpha)}+\frac{c_0}{2}\| W_{l+\gamma/2} F^{(\alpha)} \|^2_{H^s} \\
\leq &C\big( \| \mu W_{\widetilde{\gamma}} F \|_{L^1} + \| \mu F \|^2_{L^1}
+\| \mu W_{|\gamma|} F \|^{s/(s-\eta)}_{L^1} \big)\| W_{l+\gamma/2} F^{(\alpha)} \|^2_{L^2}
      \nonumber\\
\leq &C \| W_{l+\gamma/2} F^{(\alpha)} \|^2_{L^2}\nonumber\\
\leq &C \| W_l F^{(\alpha)} \|_{L^2} \| W_{l+\gamma} F^{(\alpha)} \|_{L^2},\nonumber
\end{align}
where we used the fact
$$
\big( \| \mu W_{\widetilde{\gamma}} F \|_{L^1} + \| \mu F \|^2_{L^1}
+\| \mu W_{|\gamma|} F \|^{s/(s-\eta)}_{L^1} \big) \leq C.
$$
We remark here when considering the above inequality, we could obtain, for example,
$$
\| \mu W_{\widetilde{\gamma}} F \|_{L^1} \leq C\|W_l F \|_{L^2}
=C\|f\|_{\delta-\kappa t,l,\rho,0,r} \leq C,
$$
due to Definition \ref{Def1} about the smooth Maxwellian decay solution. The rest two terms could be considered similarly.\\
In the case of $-1<\gamma+2s \leq 0$, using the similar method we obtain
\begin{align}\label{coercive-2}
&\Psi_{1,1}^{(0,\alpha)}+c_0\| W_{l+\gamma/2} F^{(\alpha)} \|^2_{H^s} \\
\leq &C\big( \| \mu W_{|\gamma|} F \|_{L^1}+\| \mu W_{|\gamma|} F \|_{L^{\frac{3}{2}}} \big)\|W_{l+\gamma/2} F^{(\alpha)}\|^2_{H^\eta} \nonumber \\
&+ C\| \mu W_{\gamma+2} F \|_{L^1} \| W_{l+\gamma/2} F^{(\alpha)} \|^2_{L^2} \nonumber\\
\leq &C \|W_{l+\gamma/2} F^{(\alpha)}\|^2_{H^\eta} + \| W_{l+\gamma/2} F^{(\alpha)} \|^2_{L^2}\,. \nonumber
\end{align}

    On the other hand, in order to estimate $\Psi_{1,2}^{(0,\alpha)}$ with $0<\gamma +2s<1$ we need the
following commutator estimate:
\begin{lemm}\label{commutator}
There exists a $C>0$ such that
\begin{align}
\Big|\big( W_lQ(g,f)-Q(g,W_l f), h \big)_{L^2} \Big| \leq C
\| g \|_{L^1_{l+\gamma}} \| f \|_{L^2_{l+\gamma}} \| h \|_{L^2}.
\end{align}
\end{lemm}
\begin{proof}
Firstly the Taylor formula yields
\begin{align*}
|W_l-W'_l|\leq &C |v-v'| W_{l-1}\big(\tau v'+(1-\tau)v \big)
\leq C\sin(\frac{\theta}{2})|v-v_*| (W_{l-1} + W'_{l-1})\\
\leq &C\theta|v-v_*| (W_{l-1} + W_{l-1,*})
\leq C\theta|v-v_*| W_{l-1} W_{l-1,*}\,\,,
\end{align*}
where $\tau \in [0,1]$ and we used Lemma 2.3 in \cite{amuxy-regulariz}. Noticing that $ \gamma +1 > \gamma +2s>0 $, we have
\begin{align*}
&\Big|\big( W_lQ(g,f)-Q(g,W_l f), h \big)_{L^2} \Big| \\
=&\Big| \iiint\! b(\cos \theta) |v-v_*|^\gamma g_* f (W_l-W'_l) h' dv dv_* d\sigma \Big| \\
\leq &C \iiint\! b(\cos \theta) \theta |v-v_*|^{\gamma+1}|(W_{l-1}g)_*|| (W_{l-1} f)|| h'| dv dv_* d\sigma\\
\leq &C \iiint\! b(\cos \theta) \theta |(W_{l+\gamma}g)_*||(W_{l+\gamma} f)|| h'| dv dv_* d\sigma\\
\leq &C \big( \iiint\! b(\cos \theta) \theta |(W_{l+\gamma}g)_*|
                                               |(W_{l+\gamma} f)|^2 dv dv_* d\sigma\big)^{1/2}\\
 &\times \big( \iiint\! b(\cos \theta) \theta |(W_{l+\gamma}g)_*|| h'|^2 dv dv_* d\sigma \big)^{1/2}\\
\leq &C \| g \|_{L^1_{l+\gamma}} \| f \|_{L^2_{l+\gamma}} \| h \|_{L^2}.
\end{align*}
\end{proof}
Now set $g=\mu F$, $f=F^{(\alpha)}$, and $h=W_l F^{(\alpha)}$. Then in the case of $0<\gamma+2s<1$ we have
\begin{align}\label{commu-1}
|\Psi_{1,2}^{(0,\alpha)}|
&\leq C\| \mu W_{l+\gamma} F \|_{L^1} \|W_{l+\gamma} F^{(\alpha)} \|_{L^2}\|W_l F^{(\alpha)} \|_{L^2}\\
&\leq C\|W_{l+\gamma} F^{(\alpha)} \|_{L^2}\|W_l F^{(\alpha)} \|_{L^2}\,.\nonumber
\end{align}

    In the case of $-1<\gamma+2s \leq 0$, we recall Corollary 2.1 in \cite{chenyemin}:\\[-1.5em]
\begin{lemm}\label{commutator-2}
Let $N_1=|N_2|+|N_3|+\max{\{|m-2|,|m-1|\}}$ and $\widetilde{N_1}=N_2+N_3$ with $N_2,\,N_3,\,m \in \RR$. Then if $\widetilde{N_1}\geq m+\gamma $ and $0<s<1/2$, one has
\begin{align}
\Big|\big( W_m Q(g,f)-Q(g,W_m f), h \big)_{L^2} \Big| \leq C
\| g \|_{L^1_{N_1}} \| f \|_{H^\eta_{N_2}} \| h \|_{H^\eta_{N_3}}.
\end{align}
where $\eta <s$.
\end{lemm}
We remark here the result of Lemma \ref{commutator} is included in Lemma \ref{commutator-2}, but the process will be more simple if we use Lemma \ref{commutator} in the case of $0<\gamma +2s<1$. \\

    In order to estimate $\Psi_{1,2}^{(0,\alpha)}$ in the case of $-1<\gamma+2s \leq 0$, using Lemma
\ref{commutator-2} with $g=\mu F$, $f=F^{(\alpha)}$, and $h=W_l F^{(\alpha)}$, and setting $m=l$, $N_2=l+\gamma/2$, $N_3=\gamma/2$, then $N_1=2l+|\gamma|-1$, and we have
\begin{align}\label{commu-2}
|\Psi_{1,2}^{(0,\alpha)}|
&\leq C \| \mu W_{2l+|\gamma|-1} F \|_{L^1} \|W_{l+\gamma/2} F^{(\alpha)} \|^2_{H^\eta} \\
&\leq C \|W_{l+\gamma/2} F^{(\alpha)} \|^2_{H^\eta}\,.\nonumber
\end{align}
Together with (\ref{coercive-2}) and (\ref{commu-2}) we obtain
\begin{align}
&\Psi_1^{(0,\alpha)}+c_0\| W_{l+\gamma/2} F^{(\alpha)} \|^2_{H^s} \\
\leq &C \|W_{l+\gamma/2} F^{(\alpha)}\|^2_{H^\eta} + \| W_{l+\gamma/2} F^{(\alpha)} \|^2_{L^2} \nonumber\\
\leq &\varepsilon  \|W_{l+\gamma/2} F^{(\alpha)}\|^2_{H^s}
   + C_\varepsilon \| W_{l+\gamma/2} F^{(\alpha)} \|^2_{L^2}\,. \nonumber
\end{align}
Choosing $\varepsilon=c_0/2$, we have
\begin{align}
\Psi_1^{(0,\alpha)}+\frac{c_0}{2}\| W_{l+\gamma/2} F^{(\alpha)} \|^2_{H^s}
&\leq C \| W_{l+\gamma/2} F^{(\alpha)} \|^2_{L^2} \\
&\leq C \|W_{l+\gamma} F^{(\alpha)}\|_{L^2} \| W_l F^{(\alpha)} \|_{L^2} \,.\nonumber
\end{align}
We mention that in the case of $0<\gamma+2s<1$ we will obtain the same estimate on $\Psi_1^{(0,\alpha)}$ as above, by means of (\ref{coercive-1}) and (\ref{commu-1}). \\
Thus we obtain with $-1<\gamma +2s<1$:
\begin{align}\label{Psi_1^0}
&\frac{\rho^{2|\alpha|}\big| \Psi_1^{(0,\alpha)}(t)\big|}{\{(\alpha-r)!\}^{2\nu}}
+c_0 \frac{\rho^{2|\alpha|} \| W_{l+\gamma/2} F^{(\alpha)} \|^2_{H^s}}{\{(\alpha-r)!\}^{2\nu}}\\
\leq &C \|f(t)\|_{\delta-\kappa t,l,\rho,\alpha,r} \|f(t)\|_{\delta-\kappa t,l+1,\rho,\alpha,r}.
       \nonumber
\end{align}

    In order to estimate the term $\Psi_1^{(\alpha',\alpha'')}\,\, (\alpha' \neq 0 )$, we decompose in the
case of $0<\gamma+2s<1$:
\begin{align*}
\Psi_1^{(\alpha',\alpha'')}=&\left( Q(\mu F^{(\alpha')},F^{(\alpha'')}), W_{2l}F^{(\alpha)} \right)\\
=&\left( Q(\mu F^{(\alpha')},W_l F^{(\alpha'')}), W_l F^{(\alpha)} \right)\\
 &+\left( W_l Q(\mu F^{(\alpha')},F^{(\alpha'')})- Q(\mu F^{(\alpha')},W_l F^{(\alpha'')}), W_l F^{(\alpha)} \right)\\
=&\Psi_{1,1}^{(\alpha',\alpha'')}+\Psi_{1,2}^{(\alpha',\alpha'')}\,\,.
\end{align*}

 As for the estimate on $\Psi_{1,2}^{(\alpha',\alpha'')}$, setting $g=\mu F^{(\alpha')}$, $f=F^{(\alpha'')}$, and $h=W_l F^{(\alpha)}$ in the Lemma \ref{commutator}, then we have
\begin{align*}
\left| \Psi_{1,2}^{(\alpha',\alpha'')} \right|
\leq &C \| \mu F^{(\alpha')} \|_{L^1_{l+\gamma}} \| F^{(\alpha'')} \|_{L^2_{l+\gamma}} \| W_l F^{(\alpha)} \|_{L^2} \\
\leq &C \| W_l F^{(\alpha')} \|_{L^2} \|W_{l+1} F^{(\alpha'')} \|_{L^2} \| W_l F^{(\alpha)} \|_{L^2}.
\end{align*}
In order to estimate $\Psi_{1,1}^{(\alpha',\alpha'')}$, we need the following upper bound estimate (see Proposition 3.6 of \cite{smooth effect}):\\[-1.5em]
\begin{lemm}\label{upper bound}
Let $\gamma+2s >0$ and $0<s<1$. For any $\sigma \in [2s-1,2s]$ and $p \in [0, \gamma +2s]$ we have
\begin{align*}
\Big| \Big( Q(f,g),h \Big)_{L^2(\RR^3)} \Big | \leq C
\|f\|_{L^1_{\gamma+2s} }\|g\|_{H^\sigma_{\gamma+2s -p}} \|h\|_{H^{2s-\sigma}_p}\,.
\end{align*}
\end{lemm}
Next, we use this lemma with $f=\mu F^{(\alpha')}$, $g=W_l F^{(\alpha'')}$, $h=W_l F^{(\alpha)}$, $p=\gamma+2s$, and $\sigma =2s$. Then by setting $\beta =1-\gamma-2s>0$, we have
\begin{align*}
\left| \Psi_{1,1}^{(\alpha',\alpha'')} \right|
\leq &C \| \mu F^{(\alpha')} \|_{L^1_{\gamma+2s}} \| W_l F^{(\alpha'')} \|_{H^{2s}} \| W_l F^{(\alpha)} \|_{L^2_{\gamma +2s}} \\
\leq &C \| W_l F^{(\alpha')} \|_{L^2} \|W_{l+1} F^{(\alpha'')} \|_{L^2} \| W_{l+1-\beta} F^{(\alpha)} \|_{L^2} \\
& + C \| W_l F^{(\alpha')} \|_{L^2} \|W_l F^{(\alpha''+1)} \|_{L^2} \| W_{l+1} F^{(\alpha)} \|_{L^2} \,,
\end{align*}
in view of $2s<1$ and
$$
\partial_v (W_l \mu^{-1}f^{(\alpha'')})=\partial_v (W_l \mu^{-1}) f^{(\alpha'')}
+ W_l\mu^{-1} f^{(\alpha''+1)}.
$$
Then
\begin{align*}
\left| \Psi_1^{(\alpha',\alpha'')} \right|
\leq &\left| \Psi_{1,1}^{(\alpha',\alpha'')} \right|+\left| \Psi_{1,2}^{(\alpha',\alpha'')} \right| \\
\leq &C\| W_l F^{(\alpha')} \|_{L^2} \|W_{l+1} F^{(\alpha'')} \|_{L^2} \| W_{l+1-\beta} F^{(\alpha)} \|_{L^2} \\
+ &C \| W_l F^{(\alpha')} \|_{L^2} \|W_l F^{(\alpha''+1)} \|_{L^2} \| W_{l+1} F^{(\alpha)} \|_{L^2} \,.
\end{align*}

    On the other hand, if $-1<\gamma+2s \leq 0$, we will estimate $\Psi_{1,1}^{(\alpha',\alpha'')}$ by
using the following result:
\begin{lemm}\label{upper bound-2}
Let $0<s<1$ and $-1<\gamma+2s \leq 0$. For any $p\in \RR$ and $m\in [s-1,s]$, there exists a $C>0$ such that
\begin{align*}
\Big| \Big( Q(f,g),h \Big)_{L^2(\RR^3)} \Big |
\leq C \big( \|f\|_{L^1_{p^++(\gamma+2s)^+}} + \|f\|_{L^{\frac{3}{2}}} \big)
\|g\|_{H^{\max{\{s+m,(2s-1+\varepsilon)^+\}}}_{(p+\gamma+2s)^+}} \|h\|_{H^{s-m}_{-p}}\,.
\end{align*}
\end{lemm}
For the proof of the lemma we refer to Proposition 2.1 and Proposition 2.9 of \cite{amuxy4-3}. We use it with $f=\mu F^{(\alpha')}$, $g=F^{(\alpha'')}$, $h=W_{2l} F^{(\alpha)}$, $p=l-\gamma-2s$, and $m=s$. By setting $\beta =1-\gamma-2s>0$, we obtain
\begin{align*}
\left| \Psi_1^{(\alpha',\alpha'')} \right|
\leq &C
\big(\| \mu F^{(\alpha')} \|_{L^1_{l-\gamma -2s+(\gamma+2s)^+}} + \| \mu F^{(\alpha')} \|_{L^\frac{3}{2}} \big)
\| W_l F^{(\alpha'')} \|_{H^{2s}}\\
&\times \| W_{l+1-\beta} F^{(\alpha)} \|_{L^2_{\gamma +2s}}\\
\leq &C\| W_l F^{(\alpha')} \|_{L^2} \|W_{l+1} F^{(\alpha'')} \|_{L^2} \| W_{l+1-\beta} F^{(\alpha)} \|_{L^2} \\
& + C \| W_l F^{(\alpha')} \|_{L^2} \|W_l F^{(\alpha''+1)} \|_{L^2} \| W_{l+1} F^{(\alpha)} \|_{L^2} \,.
\end{align*}
Notice that we have same estimate on $ \Psi_1^{(\alpha',\alpha'')} $ in the two cases of $0<\gamma+2s<1$ and $-1<\gamma+2s \leq 0$. \\
Since the H\"{o}lder inequality yields
$$
\| W_{1-\beta} G\|_{L^2}\leq \|G\|^\beta_{L^2} \|W_1 G\|^{1-\beta}_{L^2},
$$
we have
\begin{align}
\left| \Psi_1^{(\alpha',\alpha'')}(t) \right|
\leq &C\|W_l F^{(\alpha')}\|_{L^2} \|W_{l+1} F^{(\alpha'')}\|_{L^2} \|W_{l+1} F^{(\alpha)}\|^{1-\beta}_{L^2} \|W_l F^{(\alpha)}\|^\beta_{L^2} \\
    &+ C\|W_l F^{(\alpha')}\|_{L^2} \|W_l F^{(\alpha''+1)}\|_{L^2}\|W_{l+1} F^{(\alpha)}\|_{L^2}
    \nonumber\\
=&J_1^{(\alpha',\alpha'')}(t)+J_2^{(\alpha',\alpha'')}(t). \nonumber
\end{align}

    Thus under the assumption $-1<\gamma+2s<1$ we obtain the estimate for $\Psi_1^{(\alpha',\alpha'')}(t) $ with $\alpha' \neq 0$ as follows:
\begin{align}\label{J-1 estimate}
&\frac{\rho^{2|\alpha|}\big| J_1^{(\alpha',\alpha'')}(t) \big|}{\{(\alpha-r)!\}^{2\nu}}
\leq  C \frac{\{(\alpha'-r)!\}^\nu\{(\alpha''-r)!\}^\nu}{\{(\alpha-r)!\}^\nu}\\
&\times \|f(t)\|_{\delta-\kappa t,l,\rho,\alpha',r} \|f(t)\|^\beta_{\delta-\kappa t,l,\rho,\alpha,r}
        \|f(t)\|_{\delta-\kappa t,l+1,\rho,\alpha'',r}
        \|f(t)\|^{1-\beta}_{\delta-\kappa t,l+1,\rho,\alpha,r}\,,
        \nonumber
\end{align}
and
\begin{align}
\frac{\rho^{2|\alpha|}\big| J_2^{(\alpha',\alpha'')}(t) \big|}{\{(\alpha-r)!\}^{2\nu}}
&\leq  C \frac{\{(\alpha'-r)!\}^\nu\{(\alpha''+1-r)!\}^\nu}{\{(\alpha-r)!\}^\nu}\\
&\times \|f(t)\|_{\delta-\kappa t,l,\rho,\alpha',r} \|f(t)\|_{\delta-\kappa t,l,\rho,\alpha''+1,r}
        \|f(t)\|_{\delta-\kappa t,l+1,\rho,\alpha,r}\,.
        \nonumber
\end{align}

    Now we have completed the estimates for $\Psi_1^{(0,\alpha)}(t)$, $\Psi_1^{(\alpha',\alpha'')}(t)$,
and $\Psi_2^{(\alpha',\alpha'')}(t)$. The rest proof of Lemma \ref{Main Lemma} is similar as in \cite{mu}. We give a brief outline here for the self-containedness.

First we refer to Proposition 3.1 of \cite{mu}:\\[-2em]
\begin{prop}
If $\nu \geq 1$ and $2\leq r \in \NN$ then there exists a constant $B>0$ depending only on $n$ and $r$ such that for any $\alpha \in \ZZ^n$
\begin{align}
\sum_{\alpha=\alpha'+\alpha''} \frac{\alpha!}{\alpha'!\alpha''!}
      \frac{\{(\alpha'-r)!\}^\nu\{(\alpha''-r)!\}^\nu}{\{(\alpha-r)!\}^\nu} \leq B.
\end{align}
Furthermore, if $\nu>1$ and $r>1+\nu/(\nu-1)$ then there exists a constant $B'>0$ depending only on $n$, $\nu$ and $r$ such that for any $0 \neq \alpha \in \ZZ^n$
\begin{align}
\sum_{\alpha=\alpha'+\alpha'',\,\alpha' \neq 0} \frac{\alpha!}{\alpha'!\alpha''!}
      \frac{\{(\alpha'-r)!\}^\nu\{(\alpha''+1-r)!\}^\nu}{\{(\alpha-r)!\}^\nu} \leq B'.
\end{align}
\end{prop}

    Now we consider the second term of the right-hand side of (\ref{EQ2}). We consider firstly
\begin{align}
\int^t_0\! &\frac{\rho^{2|\alpha|} \mathcal{K}^\alpha(\tau)}{\{(\alpha-r)!\}^{2\nu}}d\tau
\leq C\sum \frac{\alpha!}{\alpha'!\alpha''!}
      \frac{\{(\alpha'-r)!\}^\nu\{(\alpha''-r)!\}^\nu}{\{(\alpha-r)!\}^\nu} \\
&\times \int^t_0\! \|f(\tau)\|_{\delta-\kappa \tau,l,\rho,\alpha',r}
                   \|f(\tau)\|_{\delta-\kappa \tau,l,\rho,\alpha'',r}
                   \|f(\tau)\|_{\delta-\kappa \tau,l+1,\rho,\alpha,r} d\tau
        \nonumber\\
=&C\Big( \sum_{\min(|\alpha'|,|\alpha''|)<3r} + \sum_{\min(|\alpha'|,|\alpha''|) \geq 3r} \Big)\nonumber\\
=&\mathcal{M}_1^\alpha + \mathcal{M}_2^\alpha \,. \nonumber
\end{align}
We mention that $ 3r \leq |\alpha| \leq N $. Letting
$$
A=A_{r,n}(f)=\sup_{t\in [0,T]} \max_{|\alpha'|<3r} \|f(t)\|_{\delta-\kappa t,l,\rho,\alpha',r}\,\,,
$$
we obtain
\begin{align*}
\mathcal{M}_1^\alpha \leq CB\Big( \frac{A^2}{4 \varepsilon} \int^t_0\! \|f\|^2_{l,\rho,r,N}(\tau)d\tau + \varepsilon \sup_{3r \leq |\alpha| \leq N}
   \int^t_0\! \|f(\tau)\|^2_{\delta-\kappa \tau,l+1,\rho,\alpha,r} d\tau \Big)\,,
\end{align*}
and
\begin{align*}
\mathcal{M}_2^\alpha \leq CB\Big( \frac{1}{4 \varepsilon} \int^t_0\! \|f\|^4_{l,\rho,r,N}(\tau)d\tau + \varepsilon \sup_{3r \leq |\alpha| \leq N}
   \int^t_0\! \|f(\tau)\|^2_{\delta-\kappa \tau,l+1,\rho,\alpha,r} d\tau \Big)\,,
\end{align*}
for any $\varepsilon >0$. Taking a sufficiently smaller $\varepsilon < \kappa$, we have
\begin{align}\label{K estimate}
\int^t_0\! \frac{\rho^{2|\alpha|} \mathcal{K}^\alpha(\tau)}{\{(\alpha-r)!\}^{2\nu}}d\tau
\leq &C_\kappa \int^t_0\! \big(\|f\|^2_{l,\rho,r,N}(\tau) + \|f\|^4_{l,\rho,r,N}(\tau)\big) d\tau\\
&+ \frac{\kappa}{100}\sup_{3r \leq |\alpha| \leq N}
   \int^t_0\! \|f(\tau)\|^2_{\delta-\kappa \tau,l+1,\rho,\alpha,r} d\tau \,. \nonumber
\end{align}

    As for the integral including $\mathcal{J}^\alpha$ in (\ref{EQ2}), we get
\begin{align*}
\int^t_0\! \frac{\rho^{2|\alpha|} \mathcal{J}^\alpha(\tau)}{\{(\alpha-r)!\}^{2\nu}}d\tau
\leq & \int^t_0\! \sum_{\alpha=\alpha'+\alpha''} \frac{\alpha!}{\alpha'!\alpha''!} \frac{\rho^{2|\alpha|}\big| J_1^{(\alpha',\alpha'')}(t) \big|}{\{(\alpha-r)!\}^{2\nu}}\\
&+ \int^t_0\! \sum_{\alpha=\alpha'+\alpha'',\,\alpha' \neq 0} \frac{\alpha!}{\alpha'!\alpha''!}\frac{\rho^{2|\alpha|}\big| J_2^{(\alpha',\alpha'')}(t) \big|}{\{(\alpha-r)!\}^{2\nu}}\,.
\end{align*}
We notice that the integral including
$\mathcal{J}_2^{(\alpha',\alpha'')}$ has the same estimate as the estimate (\ref{K estimate}). On the other hand, the last factor of (\ref{J-1 estimate}) is bounded by
\begin{align*}
\varepsilon &\left( \|f(t)\|^2_{\delta-\kappa t,l+1,\rho,\alpha'',r}
                 + \|f(t)\|^2_{\delta-\kappa t,l+1,\rho,\alpha,r} \right)\\
&\qquad + C_\varepsilon \|f(t)\|^{2/\beta}_{\delta-\kappa t,l,\rho,\alpha',r}
                \|f(t)\|^2_{\delta-\kappa t,l,\rho,\alpha,r},
\end{align*}
for any small $\varepsilon>0$. Then the integral including $\mathcal{J}_1^{(\alpha',\alpha'')}$ is bounded by
\begin{align}
C_\kappa &\int^t_0\! \Big(\|f\|^2_{l,\rho,r,N}(\tau) + \|f\|^{2(1+\beta)/\beta}_{l,\rho,r,N}(\tau)\Big) d\tau \\
&+ \frac{\kappa}{100}\sup_{3r \leq |\alpha| \leq N}
   \int^t_0\! \|f(\tau)\|^2_{\delta-\kappa \tau,l+1,\rho,\alpha,r} d\tau \,.\nonumber
\end{align}
So we have
\begin{align}\label{J estimate}
\int^t_0\! \frac{\rho^{2|\alpha|} \mathcal{J}^\alpha(\tau)}{\{(\alpha-r)!\}^{2\nu}}d\tau \leq
&C_\kappa \int^t_0\! \big(\|f\|^2_{l,\rho,r,N}(\tau) + \|f\|^{2(1+\beta)/\beta}_{l,\rho,r,N}(\tau)\big) d\tau\\
&+ \frac{\kappa}{100}\sup_{3r \leq |\alpha| \leq N}
   \int^t_0\! \|f(\tau)\|^2_{\delta-\kappa \tau,l+1,\rho,\alpha,r} d\tau \,, \nonumber
\end{align}
where we used the fact $4<2(1+\beta)/\beta$.
Together with (\ref{EQ2}), (\ref{Psi_1^0}), (\ref{K estimate}) and (\ref{J estimate}) we obtain finally
\begin{align}
&\|f(t)\|^2_{\delta-\kappa t,l,\rho,\alpha,r} + c_0 \int^t_0\! \frac{\rho^{2|\alpha|} \| W_{l+\gamma/2} \mu^{-1} f^{(\alpha)} \|^2_{H^s}}{\{(\alpha-r)!\}^{2\nu}}
+ 2 \kappa \int^t_0\! \|f(\tau)\|^2_{\delta-\kappa \tau,l+1,\rho,\alpha,r} d\tau \\
\leq &\|f(0)\|^2_{\delta,l,\rho,\alpha,r} + C_\kappa \int^t_0\! \Big(\|f\|^2_{l,\rho,r,N}(\tau) + \|f\|^{2(1+\beta)/\beta}_{l,\rho,r,N}(\tau)\Big) d\tau \nonumber\\
& + \frac{\kappa}{100}\sup_{3r \leq |\alpha| \leq N}
   \int^t_0\! \|f(\tau)\|^2_{\delta-\kappa \tau,l+1,\rho,\alpha,r} d\tau \,. \nonumber
\end{align}

    This leads to the desired estimate (\ref{basic inequa}) including the extra second term of the
left-hand side. This completes the proof of Lemma \ref{Main Lemma}.

\bigskip
\noindent\textbf{Acknowledgments} This work was partially
supported by NNSFC (No. 10971235), RFDP (No. 200805580014), NCET
(No. NCET-08-0579) and the key project of Sun Yat-sen University.

\end{document}